\documentclass[12pt,epsfig,amsfonts]{amsart} 
\usepackage{amsmath,amsthm,amssymb,amscd,epsfig,color}

\usepackage[hang,small,bf]{caption}
\usepackage[subrefformat=parens]{subcaption}
\usepackage{graphicx}
\usepackage{mathrsfs}
\usepackage{comment}
\usepackage{mathrsfs}
\usepackage{placeins} 
\usepackage{url}

\makeatletter
\newcommand{\tpitchfork}{%
  \vbox{
    \baselineskip\z@skip
    \lineskip-.52ex
    \lineskiplimit\maxdimen
    \m@th
    \ialign{##\crcr\hidewidth\smash{$-$}\hidewidth\crcr$\pitchfork$\crcr}
  }%
  
}
\makeatother

\newtheorem*{theorema}{Theorem A}
\newtheorem*{theoremb}{Theorem B}
\newtheorem*{theoremA4.1}{the Main Theorem3.1}
\newtheorem*{theoremc}{Theorem C}
\newtheorem*{theoremd}{Theorem D}

\newtheorem{prop}{Proposition}[section]
\newtheorem{lemma}[prop]{Lemma}

\theoremstyle{remark}
\newtheorem{remark}[prop]{Remark}

\numberwithin{equation}{section}

\newcommand{\BF}{\boldmath}

\newcommand{\Dim}{{\text{M}}}

\author{Yoshitaka Saiki, Hiroki Takahasi, James A. Yorke}
\address{Graduate School of Business Administration, Hitotsubashi University, Tokyo,
186-8601, JAPAN} 
\email{yoshi.saiki@r.hit-u.ac.jp}
\address{Keio Institute of Pure and Applied Sciences (KiPAS), Department of Mathematics,
Keio University, Yokohama,
223-8522, JAPAN} 
\email{hiroki@math.keio.ac.jp}

\address{Institute for Physical Science and Technology and the Departments of Mathematics and Physics,
University of Maryland College Park, MD 20742, USA} 
\email{yorke@umd.edu}
\date{\today}

\begin{document}
\date{\today}
\title[The twisted baker map]{The twisted baker map}

\maketitle

\begin{abstract}
  As a model to provide a hands-on, elementary understanding of 
  ``vortex dynamics'',
  we introduce a piecewise linear non-invertible map called {\it a twisted baker map}.
  We show that the set of hyperbolic repelling
  periodic points with complex conjugate eigenvalues  
  and that without complex conjugate eigenvalues 
  are simultaneously dense in the phase space. 
  We also show that these two sets  
  equidistribute with respect to the normalized Lebesgue measure, in spite of a non-uniformity in their Lyapunov exponents.
  \end{abstract}
  
\section{Introduction}
\label{sec:introduction}
A natural approach to understanding complicated systems is to 
analyze a collection of simple 
examples that retain some essential features of the complexity of the original systems.
The baker map \cite{seidel_1933} 
is one of the simplest models of chaotic dynamical systems. 
It divides the unit square into $p\geq2$ equal vertical strips,
and maps each strip to a horizontal one by squeezing it vertically by the factor $p$ and stretching it horizontally by the same factor. The horizontal strips are laid out covering the square.
Due to the squeezing and stretching, small initial errors get amplified under iteration,
and result in an unpredictable long-term behavior.
The name ``baker'' is used since the action of the map is reminiscent of the kneading dough \cite{halmos_1956}.
For fat (area-expanding), skinny (area-contracting) and heterochaos baker maps, see for example \cite{alexander_1984}, \cite{farmer_1983}  and \cite{saiki21} respectively.

The baker map and these its variants are ``anisotropic'', in that
all coordinate directions are preserved under the action of the maps. As a consequence,
all eigenvalues of periodic points are real.
 These maps
 may not be appropriate for modeling essential features of fluid turbulence,
considering Kolmogorov's universality law for small scales, 
in which vortex dynamics is statistically isotropic~\cite{kolmogorov41}.
 In fact, in numerical investigations of high dimensional systems such as those modeling fluid turbulence,  
 it is common to find periodic points with complex conjugate eigenvalues~\cite{zammert14}. 
 Taking these physical and experimental backgrounds into consideration, we need a simple model that  provides a hands-on, elementary understanding of vortex dynamics.
 
 In this paper, we introduce a piecewise linear map that models an interaction of stretching and twisting in vortex dynamics.
This map does not preserve coordinate directions,
and has periodic points of complex conjugate eigenvalues.
  Let $\Dim\geq2$ be an integer and
let \[X =\{(x_1,\ldots,x_\Dim)\in\mathbb R^\Dim\colon -1\le x_1\le 1\text{ and }0\le x_j\le 1,\ j=2,\ldots,\Dim\}.\]
We put \[
X_L=\{(x_1,\ldots,x_\Dim)\in X:x_1<0\}\quad\text{and}\quad
    X_R=\{(x_1,\ldots,x_\Dim)\in X:x_1\ge0\}.\]
Define a {\bf twisted baker map} $F\colon X\to X$ by
 \[
  F(x_1,\dots,x_\Dim)=
  \begin{cases}
    \left(1+2x_1,x_2\ldots,x_\Dim\right)\text{ on }X_L, \\
    \left(1-2x_\Dim,x_1,\ldots,x_{
    \Dim-1}\right)\text{ on }X_R.
  \end{cases}\]
The twisted baker map can be written as a composition $F=B\circ T$ of two maps:
 \begin{equation}\label{T}
  T(x_1,\dots,x_\Dim)=
  \begin{cases}
    \left(x_1,\dots,x_\Dim\right)\text{ on }X_L,\\
    \left(x_\Dim,x_1,\ldots,x_{
    \Dim-1}\right)\text{ on }X_R, 
  \end{cases}
  \end{equation}
and
 \begin{equation}\label{B}
  B(x_1,\dots,x_\Dim)=
    \left(\tau(x_1),x_2,x_3\ldots,x_\Dim\right),
  \end{equation}
where $\tau\colon[-1,1]\to [-1,1]$ is the tent map given by
\[
  \tau(x_1)=
  \begin{cases}
    1+2x_1\text{ for }x_1<0, \\
    1-2x_1\text{ for }x_1\geq0.
  \end{cases}
  \]
  
The map $T$ ``twists'' the coordinates in part of the phase space, see Figure~\ref{fig:map} for $\Dim=2$ and Figure~\ref{fig:map3d} for $\Dim\ge 3$.
Fixed points of $F$ are the point 
$(1/3,1/3,\ldots,1/3)$ in $R$ and all points in the set 
\[N=\{(x_1,\ldots,x_\Dim)\in X\colon x_1=-1\}.\]
Let $S$ denote the set of discontinuities of $F$. We have
\[S=\{(x_1,\ldots,x_\Dim)\in X\colon x_1=0\}.\] 
 If we replace the tent map $\tau$ by the map $1+2x_1$ for $x_1<0$
and $-1+2x_1$ for $x_1\geq0$, the composed map may look more similar to the baker map, with more discontinuities than $F$. All our results on $F$ stated below remain essentially the
same for this map.

\begin{figure}[tb]
    \includegraphics[width=0.8\textwidth]{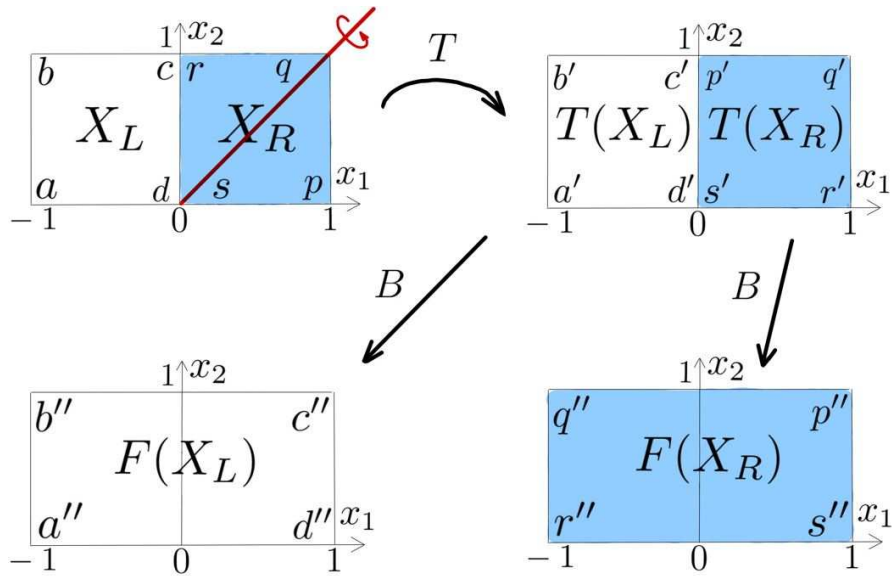}
    \caption{
     {\bf\BF The twisted baker map $F=B\circ T$ with $\Dim=2$.}
    The map $T$ twists (flips) the square $X_R$ onto itself, fixing the diagonal from $(0,0)$ to $(1,1)$. The map $B$ 
     stretches the $x_1$ coordinate of the rectangles $T(X_L)$, $T(X_R)$ onto intervals that contain $(-1,1)$, with flip for $T(X_L)$,
    leaving the other coordinate unchanged. The colored areas show $X_R$ 
    and its images under $T$ and $F$. }
    \label{fig:map}
\end{figure}
\begin{figure}[tb]
    \includegraphics[width=0.95\textwidth,height=0.7\textwidth
    ]{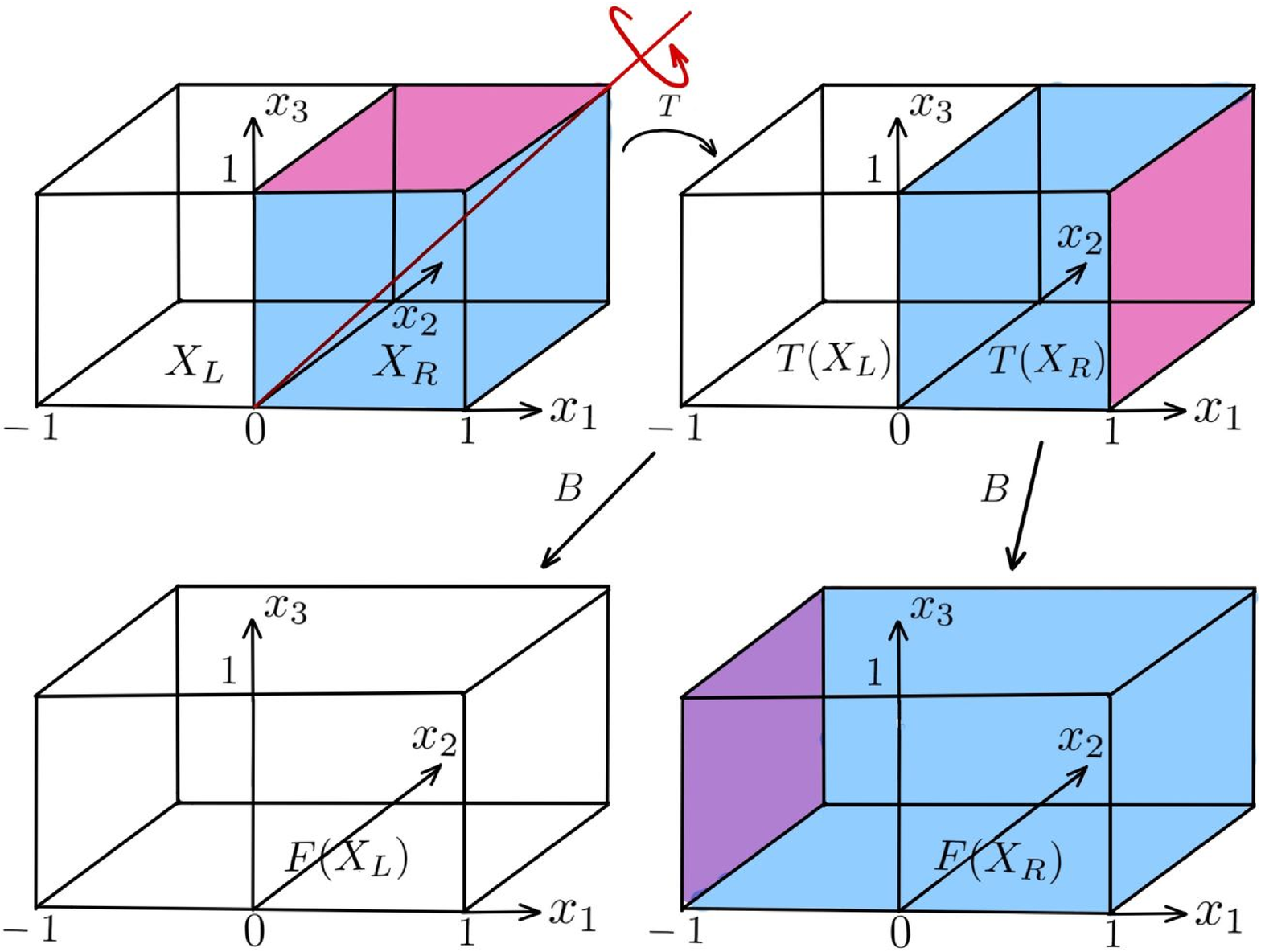}
    \caption{
    {\bf\BF The twisted baker map $F=B\circ T$ with $\Dim\geq3$.} 
    The map $T$ twists 
    the cube $X_R$ onto itself, fixing the 
    diagonal from $(0,0,\ldots,0)$ to $(1,1,\ldots,1)$. The map $B$ stretches the $x_1$ coordinate of the cubes $T(X_L)$, $T(X_R)$ onto intervals that contain $(-1,1)$, with flip for $T(X_L)$, 
    leaving other coordinates unchanged. The colored areas show $X_R$ 
    and its images under $T$ and $F$. The images of the purple
    face in the first panel are shown as an illustration of how it maps under $T$ and $F$.}
    \label{fig:map3d}
\end{figure}

For the baker map and their variants mentioned above
\cite{alexander_1984,farmer_1983,saiki21},
periodic points are dense in their maximal invariant sets. 
Our first result establishes the density of periodic points for the twisted baker map.
A new feature here is a simultaneous density of periodic points with 
complex conjugate eigenvalues and that with only real ones.
By a {\bf periodic point} of $F$ of period $n$ we mean a fixed point of $F^n$.
   A periodic point $x$ of $F$ of period $n\geq1$ is {\bf hyperbolic repelling} if  the set of eigenvalues of
  the Jacobian matrix of $F^n$ at $x$
does not intersect 
  the unit disk $\{z\in\mathbb C\colon|z|\leq1\}$. 
 If $x$ is a boundary point of $X_L$ or $X_R$, some of the first-order derivatives in the Jacobian matrix are
 taken as one-sided derivatives. 
 
For $n\geq0$ and $x\in X$,
let $JF^n(x)$
denote the Jacobian matrix of $F^n$ at $x$ with respect to the canonical basis of $\mathbb R^{\Dim}$, and let ${\rm spec}(JF^n(x))$ denote the set of eigenvalues of $JF^n(x)$.
  For each $n\geq1$ we consider three sets of periodic points of period $n$: 
 \[\begin{split}
 {\rm fix}(n)&=\{x\in X\setminus N \colon F^n(x)=x\},\\
 {\rm fix}_{\mathbb R}(n)&=\{x\in {\rm fix}(n)\colon {\rm spec}(JF^n(x))\subset\mathbb R\},\\
 {\rm fix}_{\mathbb C}(n)&=\{x\in {\rm fix}(n)\colon {\rm spec}(JF^n(x))\not\subset\mathbb R\}.
 \end{split}\]
\begin{theorema}
 Let $F\colon X\to X$ be the twisted baker map.
 \begin{itemize}
     \item[(a)] All periodic points in $\bigcup_{n=1}^\infty{\rm fix}(n)$ are hyperbolic repelling.
     \item[(b)] Both $\bigcup_{n=1}^\infty{\rm fix}_{\mathbb R}(n)$ and
     $\bigcup_{n=1}^\infty{\rm fix}_{\mathbb C}(n)$ are dense in $X$.
       \end{itemize}
 \end{theorema}
 
The twisted baker map is expanding in all directions at any periodic orbit not in $N$.
The amount of this expansion is not uniform in the following sense.
Let $x$ be 
 a hyperbolic repelling periodic point of period $\gamma(x)$. 
 The number
 \[\chi(x)=\min\left\{|\lambda|^{\frac{1}{\gamma(x)}}\colon\lambda\in {\rm spec}(JF^{\gamma(x)}(x))\right\}\]
 is strictly greater than $1$.
We say $F$ is {\bf non-uniformly expanding on periodic points} if  $\inf_{x\in U} \chi(x) =1$ 
for any non-empty open subset $U$ of $X$,
where the infimum is taken over all  hyperbolic repelling periodic points of $F$ contained in $U$.

\begin{theoremb}
The twisted baker map is non-uniformly expanding on periodic points.
 \end{theoremb}

  Proofs of the two theorems rely on
 coding of points of $X$ into sequences of $L$ and $R$ by following histories of their orbits.
        For each point $x\in X$ we associate a {\bf kneading sequence},
    a one-sided infinite sequence $\{a_n(x)\}_{n=0}^\infty$ of $L$ and $R$ by
\begin{eqnarray}\label{def a}
  a_n(x)=
  \begin{cases}
   L\text{ if }F^n(x)\in X_L \\
    R\text{ if }F^n(x)\in X_R.
  \end{cases}
  \label{eq:a_n}
\end{eqnarray}
        The map $F$
sends $X_L$ and $X_R$ affinely onto their images that contain ${\rm int}(X)$,
where ${\rm int}(\cdot)$ denotes the interior operation.
 Points whose kneading sequences coincide up to the $n$-th position form a rectangle,
 called a basic $n$-rectangle (see Section~\ref{basic}): $F^n$ is one-to-one on this rectangle and its image under $F^n$ covers ${\rm int}(X)$. A main step in the proof of Theorem~A is to show that Lebesgue almost every point in $X$ is contained in rectangles of arbitrarily small diameter. 
 Then, periodic points are coded by periodic kneading sequences, and the existence of a complex conjugate eigenvalue is partially determined by the frequency of the symbol $R$ in the kneading sequence. 
  Moreover,
 the {\bf distortion} (the ratio of the maximal and minimal side lengths) of the rectangle is  determined by 
 how $R$ appears in the kneading sequence. To prove Theorem~B, we construct sequences of periodic points  for which the appearance of $R$ is very much biased.
 
All basic $n$-rectangles fill up $X$, and all of them but the one with the kneading sequence $L^n=LL\cdots L$ ($n$-times)
contain exactly one 
 point from
 ${\rm fix}(n)$.
 Although the distortions of these rectangles are not uniform,
 the next theorem asserts that 
the sets of
 hyperbolic repelling periodic points with and without complex conjugate eigenvalues equidistribute with respect to the normalized restriction of the Lebesgue measure to $X$, denoted by $m$.
Note that $m$ is an $F$-invariant measure.
  
 \begin{theoremc}
 Let $F\colon X\to X$ be the twisted baker map
and let ${\mathbb K}={\mathbb R}$ or ${\mathbb K}={\mathbb C}$.
  For any continuous function $\phi\colon X\to\mathbb R$ we have
 \[
 \lim_{n\to\infty}\frac{1}{\#{\rm fix}_{\mathbb K}(n)}\sum_{x\in{\rm fix}_{\mathbb K}(n)}\phi(x)=\int\phi {\rm d}m.\]
  \end{theoremc}
  
One ingredient in
 a proof of Theorem~C is a uniform lower bound
 \[\liminf_{n\to\infty}
 \frac{\#{\rm fix}_{\mathbb K}(n)}{\#{\rm fix}(n)}>0,\]
which we prove 
using the formula for the
 multisection of a binomial expansion~\cite{weisstein04}.
We then show that
the coding into kneading sequences defines a continuous, measure-theoretic isomorphism between $(F,m)$ and the $(1/2,1/2)$-Bernoulli shift, for which the set of periodic points equidistribute with respect to the Bernoulli measure.

Using this isomorphism we further obtain the following result.
Recall that $(F,m)$ is {\bf mixing} 
if for all Borel subsets $A$ and $B$ of $X$ we have
\[\lim_{n\to\infty}
m(A\cap F^{-n}(B))= m(A) m(B).\]
We also say $F$ is mixing relative to $m$.
If $(F,m)$ is mixing, it is weak mixing and ergodic. 
The fat and skinny baker maps
\cite{alexander_1984,farmer_1983} are mixing relative to the natural measures,
and the heterochaos baker maps \cite{saiki21} are weak mixing relative to the Lebesgue measure.
\begin{theoremd} Let $F\colon X\to X$ be the twisted baker map. Then
$(F,m)$ is mixing.
\end{theoremd}

The rest of this paper consists of two sections and one appendix. In Section~2
 we prove preliminary results, and in Section~3 prove all the theorems.
 In the appendix
  we comment on prime periodic points of $F$, and give an alternative proof of the ergodicity of $(F,m)$.
\section{Preliminaries}
In Sections~\ref{twist-sec} and \ref{basic} we introduce two main ingredients, twist number and basic rectangles respectively,
and prove fundamental results.
\subsection{Existence of real and complex conjugate eigenvalues}\label{twist-sec}

For an integer $n\geq1$ and
 $x\in X$
we define a {\bf twist number} $t(n,x)$ by
\[t(n,x)=\#\{0\leq k\leq n-1\colon a_k(x)=R\}.\]
 Note that $t(n,x)\geq1$
 if $x\in{\rm fix}(n)$.
The twist number of a hyperbolic repelling periodic point of $F$ is related to
the existence of complex conjugate eigenvalues.

Recall that $\Dim\geq2$ is an integer and $X=[-1,1]\times[0,1]^{\Dim-1}$.
By \eqref{T} and \eqref{B}, for $x\in X_R$ we have
\begin{equation}\label{twist-jacobian}JF(x)=
\begin{cases}
  \begin{pmatrix} 0 & -2 \\
 1& 0 
 \end{pmatrix}\quad\text{for }\Dim=2,\\
  \begin{pmatrix} 
 0 &   &   &                 & -2 \\
 1 & \ddots &   & \text{\huge{0}} &    \\
   & \ddots  & \ddots &            &    \\
   &   & \ddots  &  \ddots &                 \\
 \text{\huge{0}}&  &  & 1 & 0 \\ 
  \end{pmatrix}
\quad\text{for }\Dim\geq3. 
 \end{cases}\end{equation}
By an integer multiple of $\Dim$ we mean any
product of a nonnegative integer and $\Dim$.
\begin{lemma}\label{diagonal}
Let $n\geq1$, $x\in X$ be such that $t(n,x)$ is an integer multiple of ${\rm M}$.
Then $JF^n(x)$ is a diagonal matrix. All diagonal elements have the same sign, which is positive or negative according as $t(n,x)/{\rm M}$ is even or odd.
\end{lemma}
\begin{proof}
From \eqref{twist-jacobian}, if $t(n,x)=\Dim$ then $JF^n(x)$ is a diagonal matrix whose diagonal elements are all negative, which implies
the last assertion.
\end{proof}

\begin{lemma}\label{twist-lem}
Let $n\geq1$ and
 $x\in{\rm fix}(n)$ be such that $t(n,x)$ is an integer multiple of ${\rm M}$.
Then $x\in{\rm fix}_{\mathbb R}(n)$.
\end{lemma}
\begin{proof}
Follows from Lemma~\ref{diagonal}.
\end{proof}

The next lemma allows us to find periodic points with complex conjugate eigenvalues.

\begin{lemma}\label{c-eigenvalue}
Let ${\rm M}\geq2$, and let $n\geq1$,
 $x\in{\rm fix}(n)$ satisfy $t(n,x)=q{\rm M}+r$
for an integer $q\geq0$ and $r\in\{1,{\rm M}-1\}$. Then
$x\in{\rm fix}_{\mathbb C}(n)$.
\end{lemma}
\begin{proof}
Let $M=2$. If $n=1$, then the desired conclusion
follows from a direct calculation
of the eigenvalues of the Jacobian matrix 
in \eqref{twist-jacobian}.
Let $n\ge 2$.
By periodicity, we may assume $a_{n-1}(x)=R$
with no loss of generality.
Then $t(n-1,x)$ is even.
By Lemma~\ref{diagonal} we have
\[JF^{n-1}(x)=\begin{pmatrix}
\alpha & 0\\
0& \beta
\end{pmatrix},\ \alpha\beta>0.\]
By $a_{n-1}(x)=R$ and \eqref{twist-jacobian},
\[JF(F^{n-1}(x))=\begin{pmatrix}
0 & -2\\
1& 0
\end{pmatrix}.\]
Hence
\[JF^n(x)=JF(F^{n-1}(x))JF^{n-1}(x)=\begin{pmatrix}
0 & -2\beta\\
\alpha & 0
\end{pmatrix}.\]
Therefore, ${\rm spec}(JF^n(x))=\{\pm\sqrt{-2\alpha\beta}\}$ and
$x\in{\rm fix}_{\mathbb C}(n)$.

Let $\Dim\geq3$.
From Lemma~\ref{diagonal} and
 \eqref{twist-jacobian},
we have
\[JF^n(x)=\begin{pmatrix}
O & A\\
B& O
\end{pmatrix},\]
where $A$ and $B$ are nonsingular diagonal matrices of size $r$ and $\Dim -r$  respectively.
 Since $r\in\{1,\Dim-1\}$,
the characteristic polynomial $g(\lambda)$ of $JF^n(x)$ is of the form $g(\lambda)=\lambda^\Dim-\alpha$, $\alpha\in\mathbb R\setminus\{0\}$.
 Since $\Dim\geq3$, the
 equation $g(\lambda)=0$ has a solution in $\mathbb C\setminus\mathbb R$. 
 Hence we obtain $x\in{\rm fix}_{\mathbb C}(n)$.\end{proof}

\subsection{Basic rectangles}\label{basic}
Let $n\geq1$ be an integer.
Let $\{L,R\}^n$ denote the set of words of $\{L,R\}$ with word length $n$, namely
\[\{L,R\}^n=\{a_0\cdots a_{n-1}\colon a_k\in\{L,R\}\text{ for all }0\leq k\leq n-1\}.\] 
For each $a_0\cdots a_{n-1}\in\{L,R\}^n$,
define
\[[a_0\cdots a_{n-1}]=\bigcap_{k=0}^{n-1}F^{-k}(X_{a_k}),\]
We call a subset of $X$ of this form a
{\bf\BF basic $n$-rectangle} or simply a {\bf\BF basic rectangle}.
Any basic $n$-rectangle is a Cartesian product of nondegenerate intervals, and mapped by 
$F^n$ affinely onto a rectangle that contains ${\rm int}(X)$. See Figure~\ref{fig:2D} for examples of basic rectangles.
Below we list properties of basic rectangles:
\begin{itemize}
\item[(P1)] $[a_0\cdots a_{n-1}]=[a_0\cdots a_{n-1}L]\cup[a_0\cdots a_{n-1}R]$ for all $n\geq1$ and 
$a_0\cdots a_{n-1}\in \{L,R\}^n$.

\item[(P2)] ${\rm int}(F^{n}([a_0\cdots a_{\ell-1}]))={\rm int}([a_n\cdots a_{\ell-1}])$ for all $n\geq0$, $\ell\geq n+1$ and
$a_0\cdots a_{\ell-1}\in \{L,R\}^\ell$.

\item[(P3)] Two basic rectangles are nested or disjoint.

\item[(P4)]
$m([a_0\cdots a_{n-1}])=2^{-n}$
for all $n\geq1$ and $a_0\cdots a_{n-1}\in \{L,R\}^n$.
\end{itemize}

\begin{figure}[tb]
\begin{minipage}[b]{0.475\linewidth}
    \centering
    \includegraphics[width=0.82\textwidth,height=0.45\textwidth]{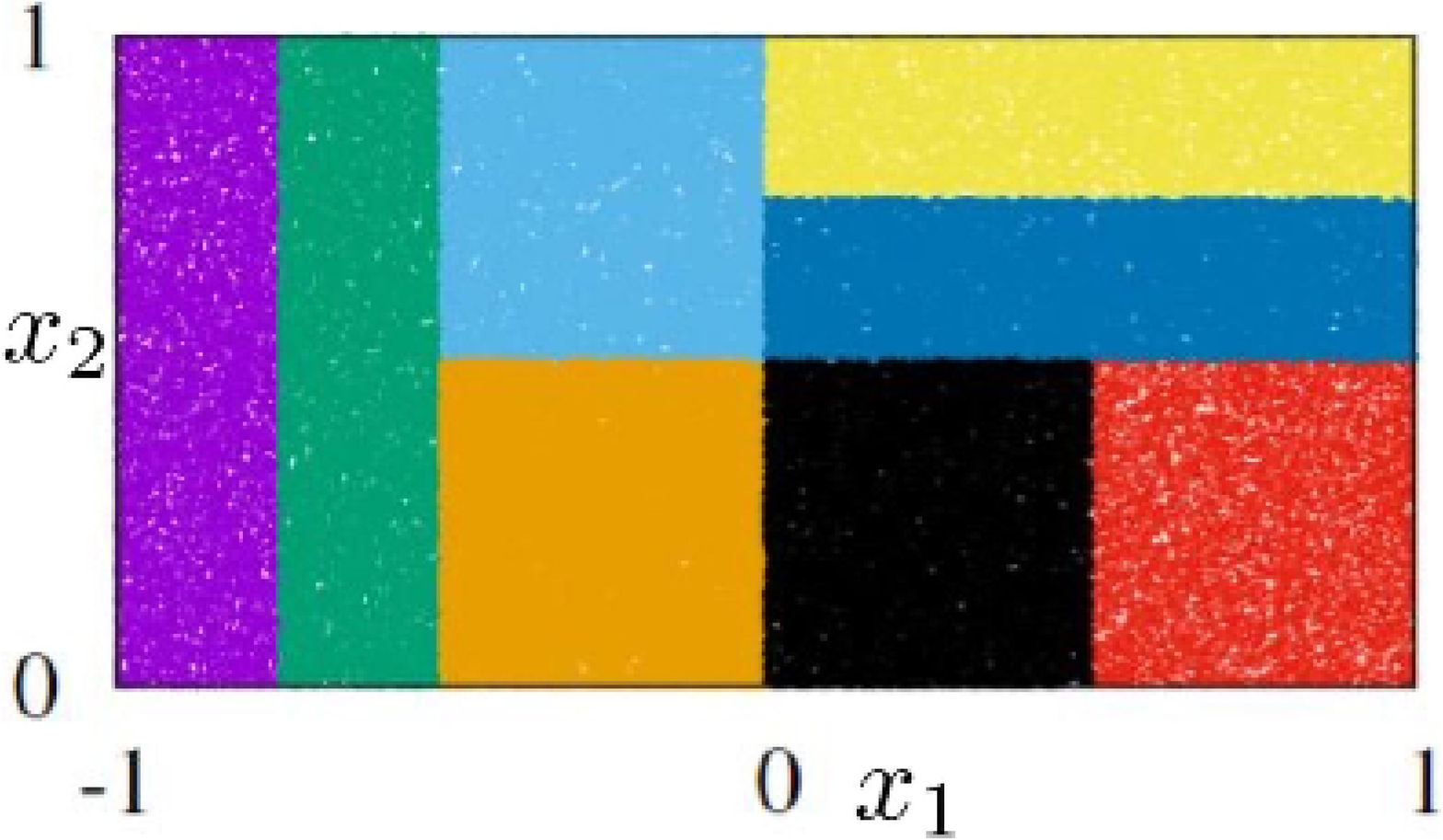}
    \subcaption{$n=3$}
  \end{minipage}
  \begin{minipage}[b]{0.475\linewidth}
    \centering
    \includegraphics[width=0.82\textwidth,height=0.45\textwidth]{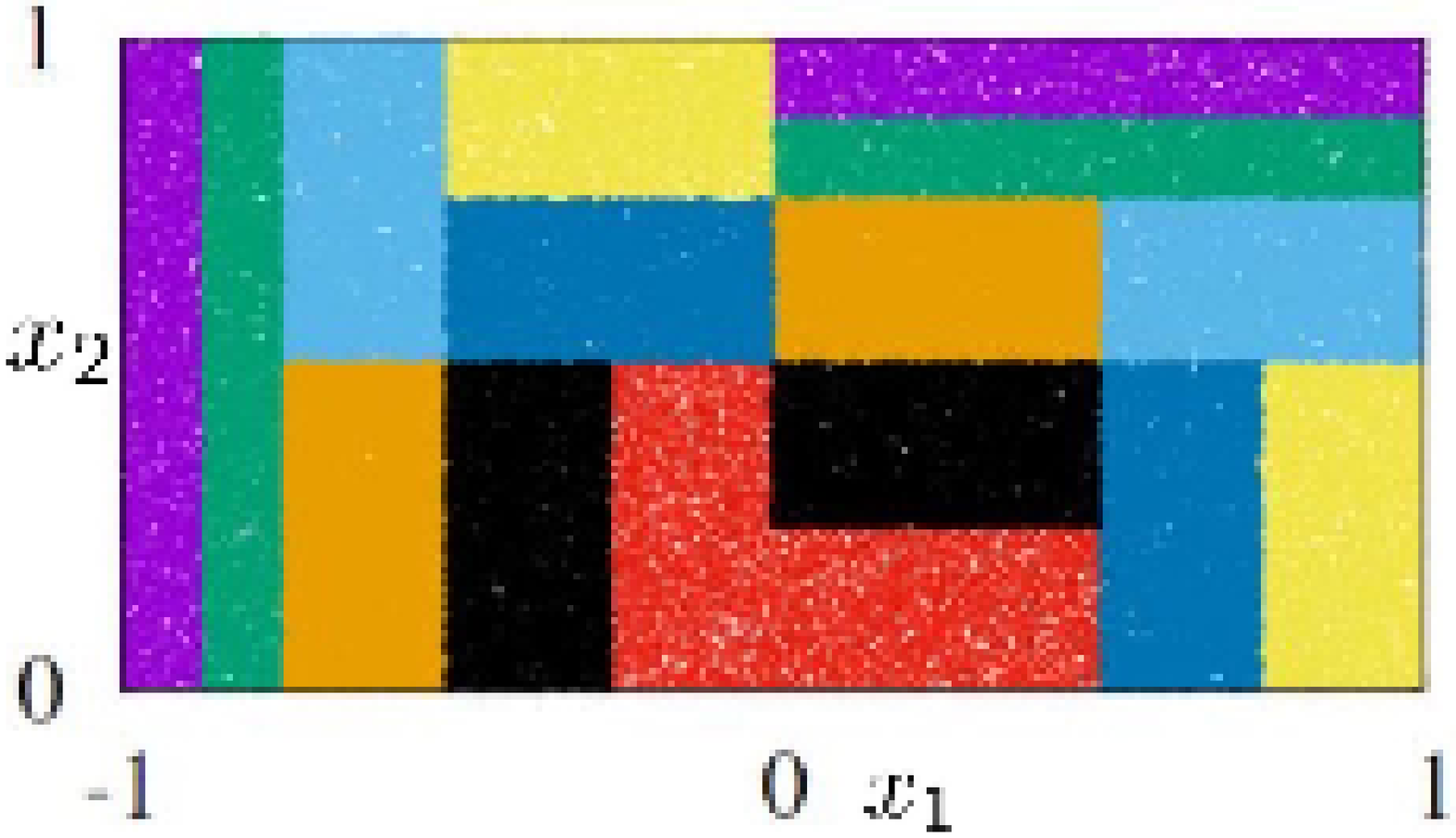}
    \subcaption{$n=4$}
  \end{minipage}
  \begin{minipage}[b]{0.475\linewidth}
    \centering
    \includegraphics[width=0.82\textwidth,height=0.45\textwidth]{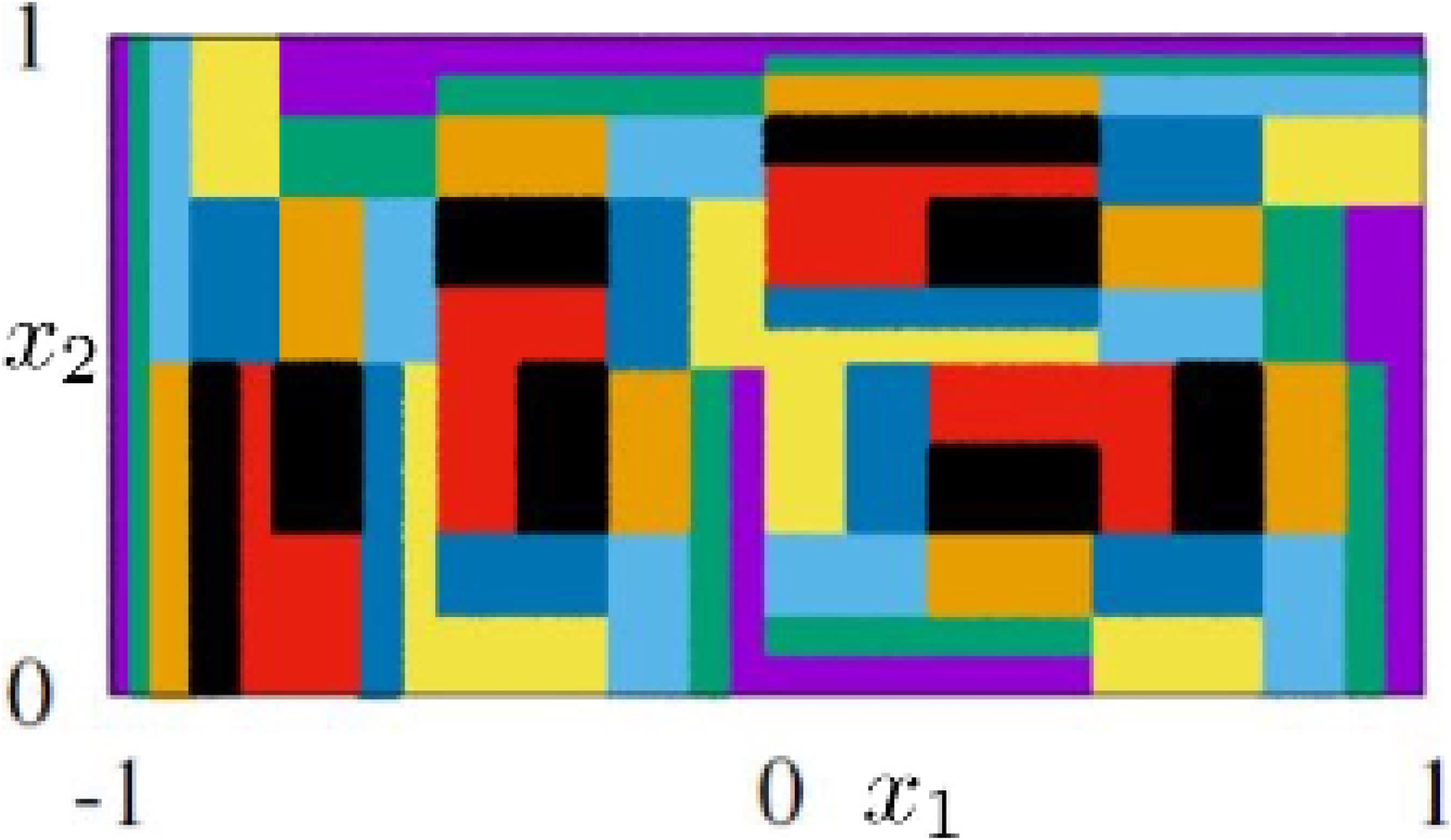}
    \subcaption{$n=6$}
  \end{minipage}
  \begin{minipage}[b]{0.475\linewidth}
    \centering
    \includegraphics[width=0.82\textwidth,height=0.45\textwidth]{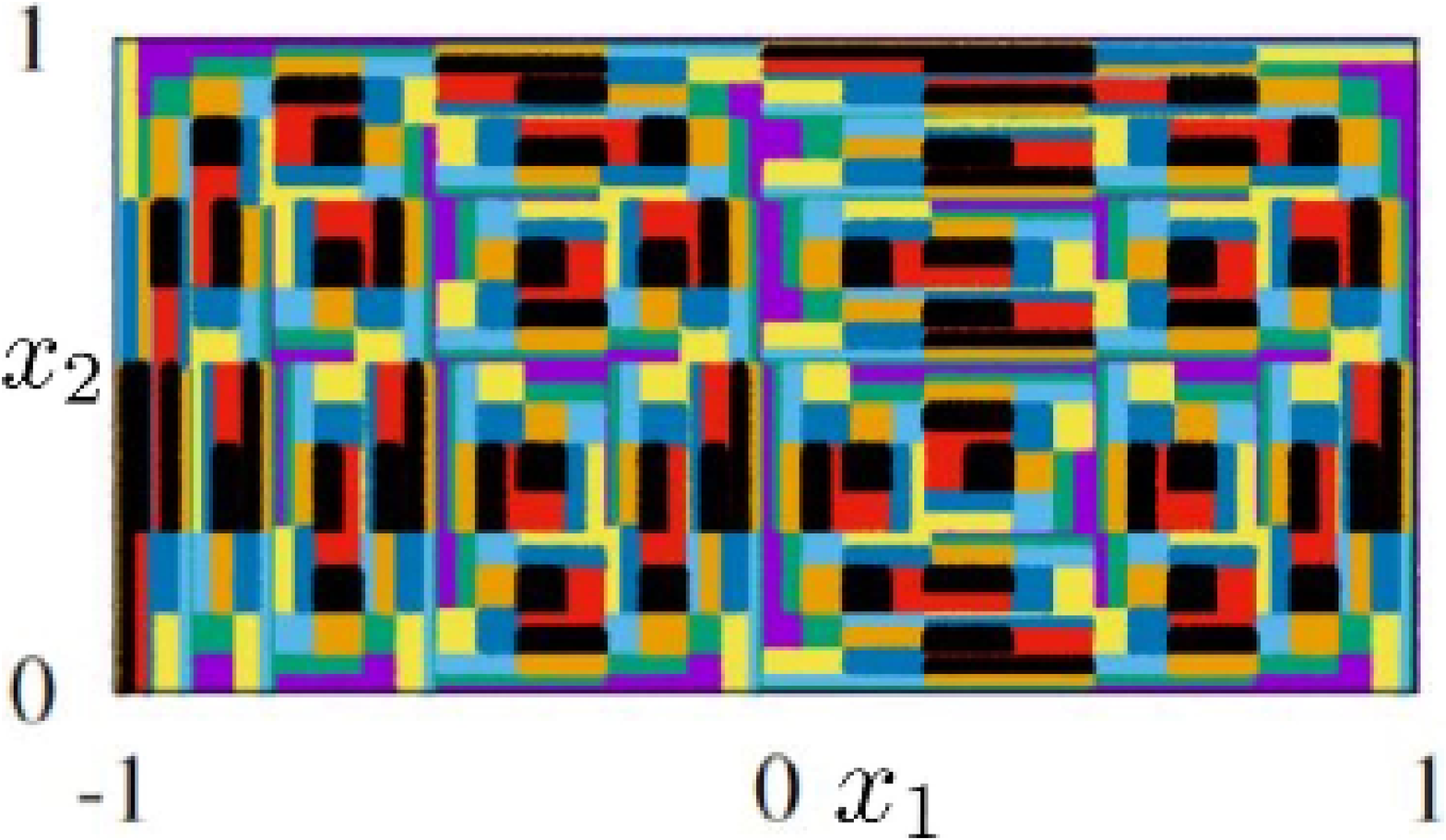}
    \subcaption{$n=9$}
  \end{minipage}
  \medskip
  \begin{minipage}[b]{0.25\linewidth}
    \centering
    \includegraphics[width=0.35\textwidth,height=0.55\textwidth]{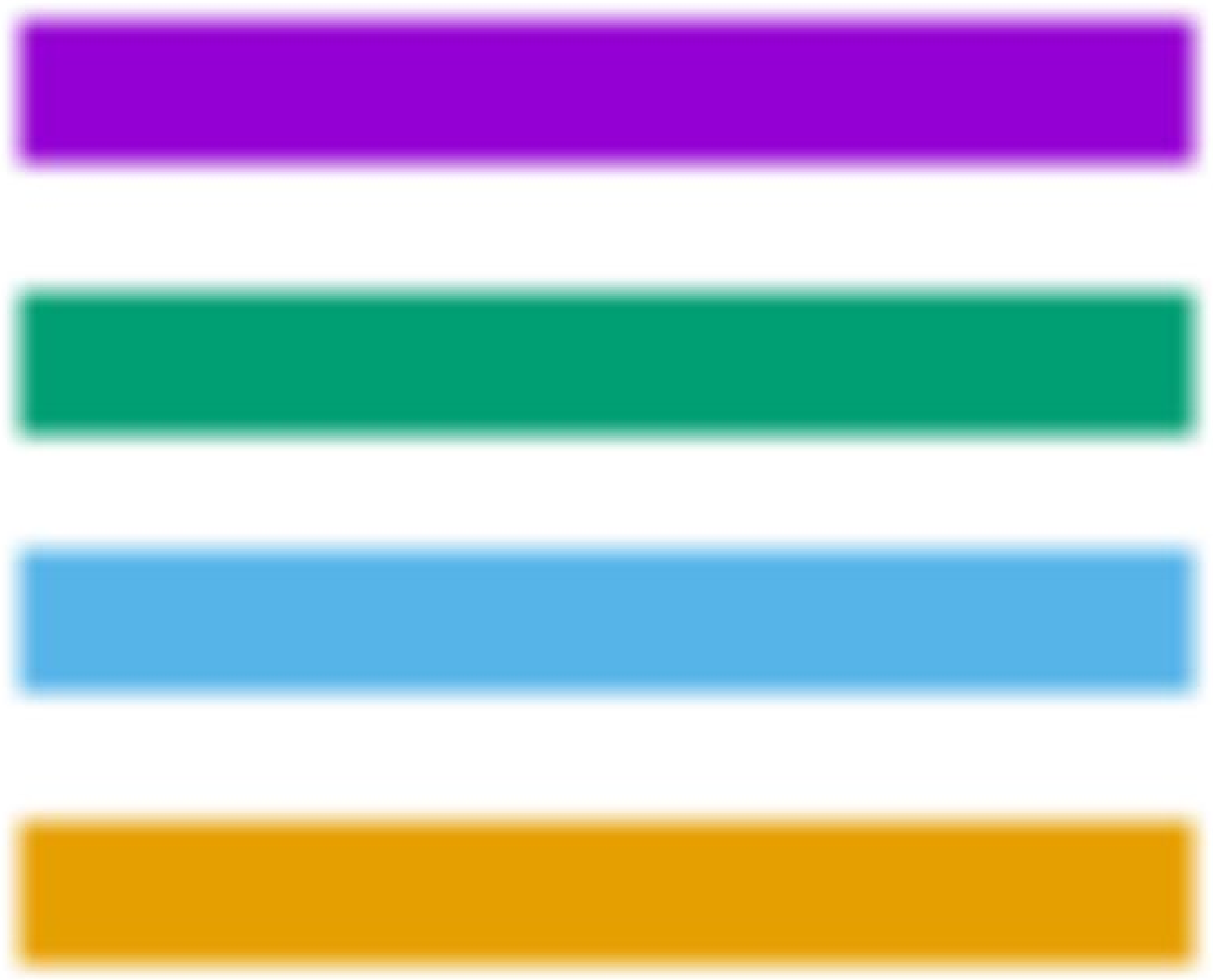}
    \includegraphics[width=0.25\textwidth,height=0.53\textwidth]{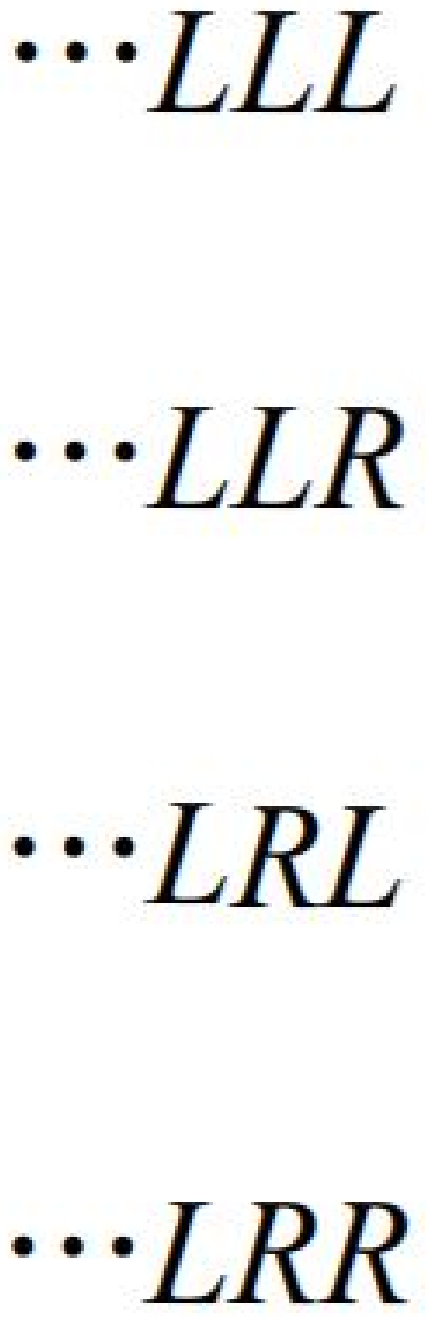}
  \end{minipage}
  \begin{minipage}[b]{0.25\linewidth}
    \centering
\includegraphics[width=0.35\textwidth,height=0.55\textwidth]{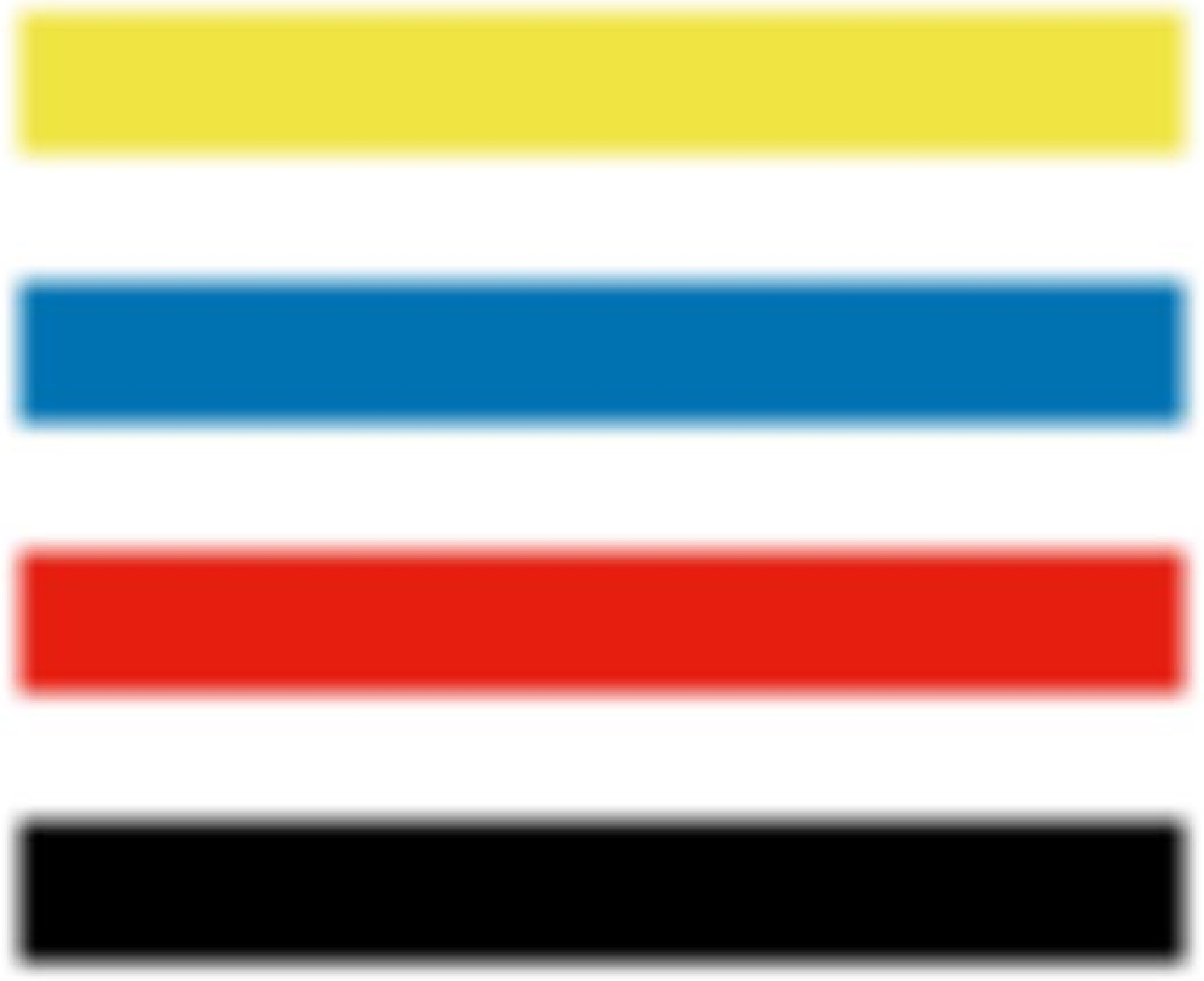}
    \includegraphics[width=0.25\textwidth,height=0.54\textwidth]{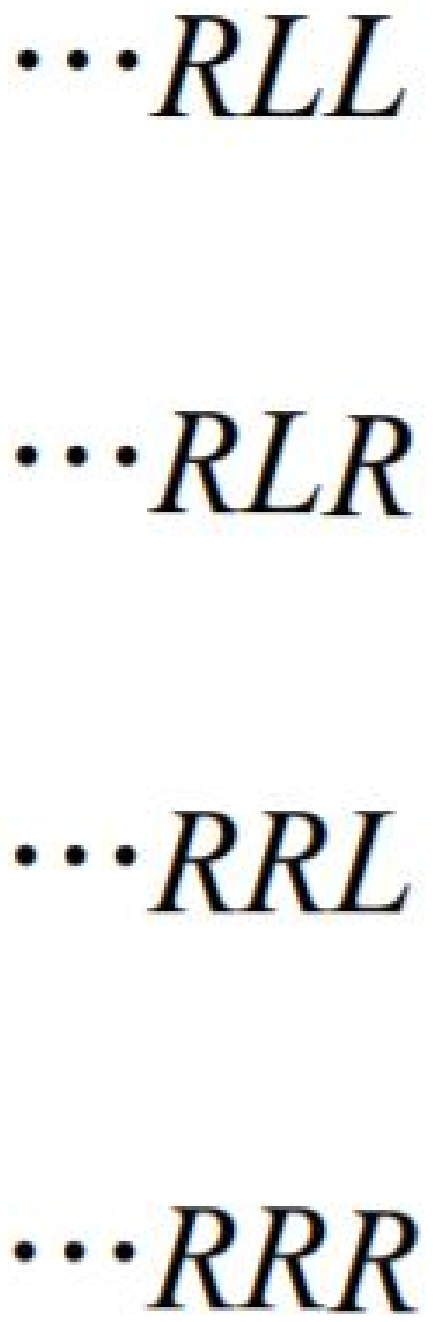}
  \end{minipage}
    \caption{
    {\bf Basic rectangles for the twisted baker map with $\Dim=2$.} 
    For $n=3,4,6,9$, basic $n$-rectangles with the same last three words have the same color, according to the coloring scheme shown above.}
    \label{fig:2D}
\end{figure}

Let $\{L,R\}^\infty$ denote the set 
of one-sided infinite sequences of $L$ and $R$.
 Let $\Sigma$ denote the set of elements of $\{L,R\}^\infty$ in
 which $R$ appears infinitely many times. 

\begin{lemma}\label{q-lemma}
For any $(a_n)_{n=0}^\infty\in \Sigma$ the following hold:

\begin{itemize}
\item[(a)] 
 $\bigcap_{n=1}^\infty [a_0\cdots a_{n-1}]$ is a singleton.
\item[(b)] 
For all sufficiently large $n\geq1$, 
there is an open subset of $X$ that 
  is compactly contained in ${\rm int}([a_0\cdots a_{n-1}])$ and contains a hyperbolic repelling periodic point of period $n$.
\end{itemize}
\end{lemma}

\begin{proof}
For an interval $I\subset[-1,1]$
let $|I|$ denote its Euclidean length.
For $0\leq k\leq n$ write
\[F^k([a_0\cdots a_{n-1}])=I_1(k,n)\times I_2(k,n)\times\cdots\times I_{\Dim}(k,n).\]
All $I_j(k,n)$ are nondegenerate intervals.
Below we show that $|I_j(0,n)|\to0$ as $n\to\infty$ for all $1\leq j\leq\Dim$. This yields (a). Then,
for all sufficiently large $n\geq1$ such that $|I_j(0,n)|<1$ for $1\leq j\leq\Dim$,
the restriction of $F^n$ to $[a_0\cdots a_{n-1}]$
expands in all directions.  
Hence there is an open subset $U$ of $X$ that is compactly contained in ${\rm int}([a_0\cdots a_{n-1}])$ and contains
a fixed point of $F^n$, which is a hyperbolic repelling periodic point of $F$, verifying (b).

Define a strictly increasing sequence $(n_i)_{i=1}^\infty$ of integers by
$n_1=0$ and
 \[n_{i+1}=\min\{n>n_i\colon \#\{n_i< k\leq n\colon a_k=R\}=\Dim\},\]
for $i=1,2,\ldots$ inductively.
Since ${\rm int}(X)\subset F^{n_i}([a_0\cdots a_{n_i-1}])$,
we have
$|I_1(n_i,n_i)|=2$ and
$|I_j(n_i,n_i)|=1$.
We have
$|I_1(n_i,n_{i+1})|\leq1$ and
$|I_j(n_i,n_{i+1})|\leq 1/2$ for $2\leq j\leq\Dim$.
Then \[\begin{split}\frac{|I_j(0,n_{i+1} )|}{|I_{j}(0,n_i)|}=\frac{|I_j(n_i,n_{i+1})|}{|I_j(n_i,n_i)|}\leq\frac{1}{2}.\end{split}\]
For the first equality we have used (P2) and the fact that $F^{n_i}|_{[a_0\cdots a_{n_i-1}]}$ is affine.
It follows that $|I_j(0,n)|\to0$ as $n\to\infty$.
\end{proof}

\begin{lemma}\label{lem-S}
There is no periodic point of $F$ in $S=\{(x_1,\ldots,x_\Dim)\in X\colon x_1=0\}$.
\end{lemma}
\begin{proof}
Suppose $x\in S$ is a periodic point of $F$. Since $x\in X_R$,
 $a_n(x)=R$ holds for infinitely many $n\geq0$, namely
 $(a_n(x))_{n=0}^\infty\in \Sigma$. By Lemma~\ref{q-lemma}(b), if $n$ is  sufficiently large and $F^n(x)=x$, 
there is an open subset of $X$ that 
  is compactly contained in   ${\rm int}([a_0(x)\cdots a_{n-1}(x)])$ and contains a hyperbolic repelling periodic point $y$ of period $n$.
  Since $x\in S$, $x$ is in the boundary of $[a_0(x)\cdots a_{n-1}(x)]$.
  Hence $x\neq y$, whereas $F^n(x)=x$ and $F^n(y)=y$. This cannot happen, since  $F^n|_{[a_0(x)\cdots a_{n-1}(x)]}$ is affine and $y$ is hyperbolic repelling.
 \end{proof}

We endow
 $\{L,R\}^\infty$ with the product topology of the discrete topology on $\{L,R\}$,
 which makes $\Sigma$ a topological subspace of $\{L,R\}^\infty$.
Let $\sigma$ denote the left shift acting on 
 $\{L,R\}^\infty$: $\sigma((a_n)_{n=0}^\infty)=(a_n)_{n=1}^\infty$. We have
$\sigma(\Sigma)=\Sigma$.
Define
 $\pi\colon \Sigma\to X$ by 
\[\pi((a_n)_{n=0}^\infty)\in\bigcap_{n=0}^\infty {\rm cl}[a_0\cdots a_n],\]
 where ${\rm cl}(\cdot)$ denotes the closure operation.
By Lemma~\ref{q-lemma}(a),
 $\pi$ is well-defined and continuous.
Moreover, $\pi$ satisfies
$F\circ \pi=\pi\circ\sigma$, and 
 is one-to-one except on the set
\[\tilde S=\bigcup_{n=0}^{\infty} F^{-n}(S),\] where it is at most $2^{\Dim}$-to-one.

 \subsection{Lower bounds on the numbers of periodic points}\label{sec-multi}
  For $n\geq1$ and $r\in\{0,1\ldots,{\rm M}-1\}$,
 let ${\rm fix}(n,r)$ denote the set of
 $x\in{\rm fix}(n)$ such that $t(n,x)=q{\rm M}+r$ holds for a non-negative integer $q\leq\lfloor n/{\rm M}\rfloor$,
 where
$\lfloor \cdot\rfloor$ denotes the floor function.

   \begin{lemma}\label{proportion}
   For any $r\in\{0,1\ldots,{\rm M}-1\}$, we have
   \[\lim_{n\to\infty}
   \frac{\#{\rm fix}(n,r)}{\#{\rm fix}(n)}=\frac{1}{{\rm M}}.\]\end{lemma}

\begin{proof}
For $n\geq1$, let $\Sigma_n$ denote the set of periodic points of period $n$ of the left shift $\sigma$.
Lemma~\ref{lem-S} implies
$\Sigma_n\cap\pi^{-1}(\tilde S)=\emptyset$, and
so
\begin{equation}\label{eq-new1}\#{\rm fix}(n)=\#\Sigma_n-1=2^n-1.\end{equation}
A similar reasoning shows
\begin{equation}\label{eq-new2}\#{\rm fix}(n,r)=\sum_{\stackrel{0\leq q\leq\lfloor n/\Dim\rfloor}{q\Dim+r\leq n}}\begin{pmatrix}n\\q\Dim+r\end{pmatrix}-1.\end{equation}

The multisection of a binomial expansion $(1+x)^n=\sum_{k=0}^n\left(\begin{smallmatrix}n\\k\end{smallmatrix}\right)x^k$ at $x=1$ in \cite{weisstein04} gives the following identity for the sum of binomial coefficients with step $\Dim$:
\begin{equation}\label{binomial}\sum_{\stackrel{0\leq q\leq\lfloor n/\Dim\rfloor}{q\Dim+r\leq n}}\begin{pmatrix}n\\q\Dim+r\end{pmatrix}=\frac{1}{\Dim}\sum_{k=0}^{\Dim-1}\left(2\cos\frac{\pi k}{\Dim}\right)^n\cos\frac{\pi(n-2r) k}{\Dim}.\end{equation} 

Put $\alpha=\cos(\pi/\Dim).$ We have
\[\left|\sum_{k=1}^{\Dim-1}\left(2\cos\frac{\pi k}{\Dim}\right)^n\cos\frac{\pi nk}{\Dim}\right|\leq (\Dim-1)(2\alpha)^n,\]
and hence
\[\begin{split}&\left|\frac{1}{2^n\Dim}\sum_{k=0}^{\Dim-1}\left(2\cos\frac{\pi k}{\Dim}\right)^n\cos\frac{\pi(n-2r)k}{\Dim}-\frac{1}{\Dim}\right|\\
&=
\left|\frac{1}{2^n\Dim}\sum_{k=1}^{\Dim-1}\left(2\cos\frac{\pi k}{\Dim}\right)^n\cos\frac{\pi (n-2r)k}{\Dim}\right|\leq\frac{\Dim-1}{\Dim}\alpha^{n}.\end{split}\] 
Since $\Dim\geq2$ we have $\alpha\in[0,1)$, and so the last number
 converges to $0$ as $n\to\infty$.
 This together with \eqref{eq-new1}, \eqref{eq-new2} yield
$\#{\rm fix}(n,r)/\#{\rm fix}(n)\to1/\Dim$ as $n\to\infty$. 
\end{proof}

 \begin{prop}\label{ratio-prop}
  For any ${\rm M}\geq2$, we have
  \[ \liminf_{n\to\infty}\frac{\#{\rm fix}_{\mathbb R}(n)}{\#{\rm fix}(n)}\geq \frac{1}{{\rm M}}\quad\text{and}\quad
  \liminf_{n\to\infty}\frac{\#{\rm fix}_{\mathbb C}(n)}{\#{\rm fix}(n)}\geq\frac{\min\{{\rm M}-1,2\}}{{\rm M}}.\]\end{prop}
  \begin{proof}
  Lemma~\ref{twist-lem} gives ${\rm fix}_{\mathbb R}(n)\supset{\rm fix}(n,0) $, and Lemma~\ref{c-eigenvalue} gives ${\rm fix}_{\mathbb C}(n)\supset{\rm fix}(n,1)\cup {\rm fix}(n,\Dim -1)$.
  Hence, the desired inequalities 
   follow from  Lemma~\ref{proportion}.
  \end{proof}

\section{Proofs of the main results }
\label{sec twisted baker proof}

In this section we prove all the theorems. 
In what follows,
we shall exclude from our consideration the set
\[\tilde N=\bigcup_{n=0}^\infty F^{-n}(N),\]
which is precisely
the set of points in $X$ whose kneading sequence contains only finitely many $R$.
Note the identity 
\[\{(a_n(x))_{n=0}^\infty\in\{L,R\}^\infty\colon x\in \tilde N\}=\{L,R\}^\infty\setminus\Sigma.\]

\subsection{Proof of Theorem~A}
Theoreom~A(a) is a consequence of Lemma~\ref{q-lemma}.
To prove Theorem~A(b),
let $x\in X$, let $U$ be an open subset of $X$ containing $x$ and
let $y\in U\setminus \tilde N$.
By Lemma~\ref{q-lemma},
there exist $n\geq1$
such that  $[a_0(y)\cdots a_{n-1}(y)]\subset U$, and a periodic point 
 $z$ of $F$ of period $n$
in $[a_0(y)\cdots a_{n-1}(y)]$. 
Write
$t(n,y)=t(n,z)=q\Dim +r$ where $q,r$ are integers with $q\geq0$  and $0\leq r\leq \Dim-1$.
 In view of (P1), set 
$\Omega=[a_0(y)\cdots a_{n-1}(y)R]$.

If $z\in{\rm fix}_{\mathbb R}(n)$,
then
let $w$ denote the fixed point of $F^{n+\Dim-r+1}$ in $\Omega$ such that $a_k(w)=R$
for all $n+1\leq k\leq n+\Dim-r+1$. We have $t(n+\Dim-r+1,w)=t(n,z)+\Dim-r+1=(q+1)\Dim+1$,
and so 
$w\in{\rm fix}_{\mathbb C}(n+\Dim-r+1)$ 
by Lemma~\ref{c-eigenvalue}.
If $z\in{\rm fix}_{\mathbb C}(n)$,
then $r\neq0$ by Lemma~\ref{twist-lem}.
Let $w$ denote the fixed point of $F^{n+\Dim-r}$ in $\Omega$
such that $a_k(w)=R$
for all $k$ with $n+1\leq k\leq n+\Dim-r$.
Then $t(n+\Dim-r,w)=t(n,z)+\Dim-r=(q+1)\Dim$, and so $w\in{\rm fix}_{\mathbb R}(n+\Dim-r)$
by Lemma~\ref{twist-lem}.
Since $U$ is an arbitrary open subset of $X$ containing $x$,  we have verified that $x$ is contained in the closure of
$\bigcup_{n=1}^\infty{\rm fix}_{\mathbb R}(n)$ and that of $\bigcup_{n=1}^\infty{\rm fix}_{\mathbb C}(n)$.
Since $x\in X$ is arbitrary, Theorem~A(b) holds. \qed

\subsection{Proof of Theorem~B}
Fix a sequence $(p_n)_{n=1}^\infty$ of positive integers such that
$\sum_{k=1}^n p_k$ is an integer multiple of $\Dim$ for all $n\geq1$ and
$\lim_{n\to\infty}\sum_{k=1}^{n-1} p_k/p_n=0$.
By Lemma~\ref{q-lemma}, for all sufficiently large $n$ 
there exists a periodic point of period of $\sum_{k=1}^n p_k$ that is contained in 
$[L^{p_1-1}RL^{p_2-1}R\cdots L^{p_n-1}R]$ which we denote by $x_n$.
Note that $t(x_n, \sum_{k=1}^n p_k)=n$.
Suppose that $n$ is an integer multiple of $\Dim$.
By Lemma~\ref{diagonal}, the matrix $JF^{\sum_{k=1}^{n} p_k}(x_{n})$ is diagonal:
\[JF^{\sum_{k=1}^{n} p_k}(x_{n})=\begin{pmatrix} \lambda_1^{(n)} & &\\
& \lambda_2^{(n)} &\\
& & \ddots\\
& & & \lambda_n^{(n)}\end{pmatrix}.\]
Since $\sum_{k=1}^np_k$ is an integer multiple of $\Dim$, the expansion by $F^{p_n-1}$ by factor $2^{p_n-1}$ corresponding to $L^{p_n-1}$ occurs in the second coordinate.
Hence, for all $j\in\{1,\ldots,\Dim\}\setminus\{2\}$ we have
$1\leq|\lambda_j^{(n)}|\leq  2^{\sum_{k=1}^{n-1}p_k+1},$
and \[1\leq\chi(x_{n})\leq
2^{\left(\sum_{k=1}^{n-1} p_k+1\right)/\sum_{k=1}^{n} p_k  }.\] 
It follows that 
$\liminf_{n\to\infty}\chi(x_{n})=1.$

Let $U$ be a non-empty open subset of $X$ and
let $x\in U\setminus \tilde N$.
From Lemma~\ref{q-lemma}, there exist an integer $\ell\geq1$ 
and a periodic point $z_n$ of period $\ell+\sum_{k=1}^n p_k$ in 
 $[a_0(x)\cdots a_{\ell-1}(x)L^{p_1-1}RL^{p_2-1}R\cdots L^{p_n-1}R]$
such that $t(x,\ell)$ is an integer multiple of $\Dim$ and $[a_0(x)\cdots a_{\ell-1}(x)]\subset U$.
Note that $t(z_n, \ell+\sum_{k=1}^n p_k)=t(x,\ell)+n$,
$JF^{\ell+\sum_{k=1}^n p_k}(z_n)=JF^{\sum_{k=1}^n p_k}(F^\ell(z_n))JF^\ell(z_n)$ 
and $JF^\ell(z_n)$ is a diagonal matrix by Lemma~\ref{diagonal}.
If $n$ is an integer multiple of $\Dim$, then by Lemma~\ref{diagonal},
 $JF^{\sum_{k=1}^np_k}(F^{\ell}(z_n))$ is a diagonal matrix.
It follows that $\liminf_{n\to\infty}
\chi(z_{n})=\liminf_{n\to\infty}
\chi(x_{n})=1$. Since $U$ is an arbitrary open subset of $X$,  $F$ is non-uniformly expanding on periodic points.
\qed

\subsection{Proofs of Theorems~C and D}
From Lemma~\ref{q-lemma}(a), the partition $\{X_L,X_R\}$ of $X$ generates the Borel sigma-field on $X$.
By (P4), $m(X\setminus \tilde N)=1$ and the Shannon-McMillan-Breiman theorem, 
the (measure-theoretic) entropy of $m$ with respect to $F$
 is equal to
\[-\int\lim_{n\to\infty}\frac{1}{n} \log m([a_0(x)\cdots a_{n-1}(x)]){\rm d}m(x)=\log2.\]

 For $Y\in\{X,\Sigma,\{L,R\}^\infty\}$
 let $\mathcal M(Y)$ denote the space of Borel probability measures on $Y$ endowed with the weak* topology.
 For $y\in Y$ let
 $\delta_y\in \mathcal M(Y)$ denote the unit point mass at $y\in Y$.

 Let $\mu_{\rm max}$ denote  the $(1/2,1/2)$-Bernoulli measure on $\{L,R\}^\infty$.
 Since the restriction of $\pi$ to the set $\Sigma\setminus\pi^{-1}(\tilde S)$ has a continuous inverse $x\mapsto (a_n(x))_{n=0}^\infty$, 
 $\pi$ induces a measurable bijection between $\Sigma\setminus\pi^{-1}(\tilde S)$
and $\pi(\Sigma)\setminus \tilde S$.
Therefore, for any $F$-invariant measure $\nu\in\mathcal M(X)$ with $\nu(\tilde S)=0$, 
there exists a $\sigma|_{\Sigma}$-invariant measure $\mu\in\mathcal M(\Sigma)$ 
such that  $\nu=\mu\circ\pi^{-1}$ and 
 the entropy of $\mu$ with respect to $\sigma|_{\Sigma}$ equals the entropy of $\nu$ with respect to $F$.

  Define  $\pi_*\colon\mathcal M(\Sigma)\to
\mathcal M(X)$ by $\pi_*\mu=\mu\circ\pi^{-1}$. 
The measure $\mu_{\rm max}$ is the unique measure of maximal entropy of $\sigma$ and the entropy is $\log2$.
 Since the entropy of $m$ with respect to $F$ is $\log 2$ and $m(\tilde S)=0$, 
we obtain
$\pi_*\mu_{\rm max}|_\Sigma=m$
where $\mu_{\rm max}|_\Sigma$ denotes the restriction of $\mu_{\rm max}$ to $\Sigma$.
In other words,
$(F,m)$ is isomorphic to $(\sigma,\mu_{\rm max})$. 
Since the latter is mixing, so is $(F,m)$ and the proof of
 Theorem~D is complete.

For ${\mathbb K}={\mathbb R}$ or ${\mathbb K}={\mathbb C}$ and 
$n\geq1$, 
consider measures 
\[\nu_{n,{\mathbb K}}=\frac{1}{\#{\rm fix}_{{\mathbb K}}(n)}\sum_{x\in{\rm fix}_{{\mathbb K}}(n)}\delta_x\in\mathcal M(X),\quad
\nu_n=\frac{1}{\#{\rm fix}(n)}\sum_{x\in{\rm fix}(n)}\delta_x\in\mathcal M(X),\]
 \[\quad\text{and}\quad\mu_n=\frac{1}{\#{\rm fix}'(n)}\sum_{x\in{\rm fix}'(n)}\delta_x\in\mathcal M(\Sigma),\]
where ${\rm fix}'(n)=\{x\in\Sigma\colon\sigma^nx=x\}$.
Since the sequence $\{2^{-n}\sum_{x\in\{L,R\}^\infty,\sigma^nx=x}\delta_x\}_{n=1}^\infty$
 in $\mathcal M(\{L,R\}^\infty)$
 converges to $\mu_{\rm max}$ as $n\to\infty$ in the weak* topology and all 
 but one periodic points of $\sigma$ 
 are contained in $\bigcup_{n=1}^\infty{\rm fix}'(n)$, the sequence
 $(\mu_n)_{n=1}^\infty$ converges to $\mu_{\rm max}|_{\Sigma}$ as $n\to\infty$ in the weak* topology. Since $\pi_*$ is continuous,
  $\pi_*\mu_{\rm max}|_{\Sigma}=m$ and $\pi_*\mu_n=\nu_n$, the sequence $(\nu_n)_{n=1}^\infty$ converges to $m$ as $n\to\infty$
in the weak* topology.

Write \begin{equation}\label{decomposition}\nu_n=\frac{\#{\rm fix}_{\mathbb R}(n)}{\#{\rm fix}(n)}\nu_{n,{\mathbb R}}+\frac{\#{\rm fix}_{\mathbb C}(n)}{\#{\rm fix}(n)}\nu_{n,{\mathbb C}}.\end{equation}
Since the space $\mathcal M(X)$ is compact, from any subsequence of the sequence of positive integers
we can choose a subsequence $(n_k)_{k=1}^\infty$ such that both $(\nu_{n,{\mathbb R}})_{n=1}^\infty$ and $(\nu_{n,{\mathbb C}})_{n=1}^\infty$ converge
to measures $\nu_{\mathbb R}$ and $\nu_{\mathbb C}$ respectively.
From Proposition~\ref{ratio-prop}, letting $k\to\infty$ in \eqref{decomposition} we obtain
$m=\rho\nu_{\mathbb R}+(1-\rho)\nu_{\mathbb C}$, $\rho\in(0,1)$.
The ergodicity of $(F,m)$ yields $\nu_{\mathbb R}=\nu_{\mathbb C}=m$.
We have verified that 
any subsequence of
$(\nu_{n,{\mathbb K}})_{n=1}^\infty$  has a subsequence that converges to $m$ in the weak* topology. Therefore,
$(\nu_{n,{\mathbb K}})_{n=1}^\infty$ 
converges in the weak* topology to $m$.
The proof of Theorem~C is complete.
\qed

\begin{remark}
The proof of Theorem~D shows that $(F,m)$ is Bernoulli. Hence it is exact, and mixing.
\end{remark}

\appendix
\def\thesection{\Alph{section}}
\section{}

\subsection{Prime periodic points}\label{prime}
Let $\tilde{\gamma}(x)$ denote the prime period of a hyperbolic repelling periodic point $x$.
  For $\mathbb K=\mathbb R$ or $\mathbb K=\mathbb C$ and $n\geq1$, let 
 \[{\rm \widetilde{fix}}_{\mathbb K}(n)=\{x\in {\rm {fix}}_{\mathbb K}(n)\colon \tilde{\gamma}(x)=n\}.\]
 We have
 ${\rm {fix}}_{\mathbb K}(n)=\sum_{d\geq1\colon d|n}\theta(n/d){\rm \widetilde{fix}}_{\mathbb K}(d)$ where $\theta$ denotes the M\"obius function.
 Using the M\"obius inversion formula and Proposition~\ref{ratio-prop} we obtain
\[\liminf_{n\to\infty}
\frac{\#{\rm \widetilde{fix}}_{\mathbb R}(n)}{\#{\rm {fix}}(n)}>0\quad\text{and}\quad\liminf_{n\to\infty}\frac{\#{\rm \widetilde{fix}}_{\mathbb C}(n)}{\#{\rm {fix}}(n)}>0.\]
Using these equalities and the part of the proof of Theorem~C, we obtain
  \[\lim_{n\to\infty}\frac{1}{\#{\rm \widetilde{fix}}_{\mathbb K}(n)}\sum_{x\in{\rm \widetilde{fix}}_{\mathbb K}(n)}\phi(x)=\int\phi {\rm d}m\]
for any continuous function $\phi\colon X\to\mathbb R$.

\subsection{Ergodicity}
We give an alternative proof of the ergodicity of $(F,m)$ that is based on the next lemma.
\medskip

\begin{lemma}\label{ergodic-lem} 
Let $A$ be a Borel subset of $X$ satisfying $F^{-1}(A)=A$.
For each nonempty open set $U$ of $X$ we have $m(A\cap U)=m(A)m(U)$.
\end{lemma}

\begin{proof}
 Lemma~\ref{q-lemma}(a) implies that
$U\setminus \tilde N$ is contained in the union of countably many basic rectangles in $U$. 
If there are two overlapping basic rectangles which take part in this union,
one contains the other by (P3).
Hence, there exists a countable collection $\mathscr{R}$ 
of pairwise disjoint basic rectangles in $U$ such that 
$U\setminus \tilde N \subset \bigcup_{\Omega\in \mathscr{R}} \Omega.$ 
Moreover, we have 
$$
m(U\setminus \tilde N)=m(U)=\displaystyle\sum_{\Omega \in \mathscr{R}} m(\Omega).
$$

Suppose that $\Omega\in \mathscr{R}$ is a basic $n$-rectangle.
Using $F^{-n}(A)=A$ we have
\[\frac{m(A\cap\Omega)}{m(\Omega)}=\frac{m(F^{n}(A\cap\Omega))}{m(F^n(\Omega))}
=\frac{m(F^n(F^{-n}(A)\cap\Omega))}{m(X)}=m(A).\]
Multiplying $m(\Omega)$ and
summing the result over all $\Omega\in\mathscr{R}$ yields the desired equality.
\end{proof}

\begin{prop}\label{ergodic-prop}
The twisted baker map $F$ is ergodic relative to $m$.\end{prop}
\begin{proof}
If $(F,m)$ is not ergodic,
 there exists a Borel subset $A$ of $X$ such that $0<m(A)<1$ and $F^{-1}(A)=A$.
 By the Lebesgue density theorem,
 $m$-almost every point of $A$ is a Lebesgue density point of $A$.
 Let $x\in A$ be a Lebesgue density point of $A$ in ${\rm int}(X)$.
 For $\varepsilon>0$ let $B_\varepsilon$ denote the open ball of radius $\varepsilon$ about $x$. If $\varepsilon$ is sufficiently small, $B_\varepsilon$ is contained in $X$, and we have
$\lim_{\varepsilon\to0}m(B_\varepsilon\cap A)/m(B_\varepsilon)=1$. Taking $B_\varepsilon = U$ in Lemma~\ref{ergodic-lem} we obtain $m(B_\varepsilon\cap A)/m(B_\varepsilon)=m(A)$.
Since $m(A)<1$ by hypothesis, we reach a contradiction.
\end{proof}

\subsection*{Acknowledgments}
We thank anonymous referees for their careful readings of the manuscript and giving useful comments. 
This research was supported by the JSPS KAKENHI
19KK0067, 19K21835, 20H01811.


\begin{thebibliography}{99}

\bibitem{alexander_1984}
Alexander J C  and Yorke J A 1984 
 Fat baker's transformations
 {\em Ergod. Theor. Dynam. Syst.} {\bf 4} 1--23

\bibitem{farmer_1983}
Farmer J D, Ott E and Yorke J A 1983
\newblock The dimension of chaotic attractors
\newblock {\em Physica D} {\bf 7} 153--180

\bibitem{halmos_1956}
Halmos P R 1956
\newblock {\em Lectures on Ergodic Theory}
\newblock (Chelsea Publishing Company, New York, NY) 

\bibitem{kolmogorov41}
Kolmogorov A N 1941
\newblock The local structure of turbulence in incompressible viscous fluid for
  very large {R}eynolds numbers
\newblock {\em Cr Acad. Sci. URSS} {\bf 30} 301--305

\bibitem{saiki21}
Saiki Y, Takahasi H and Yorke J A 2021
\newblock Piecewise linear maps with heterogeneous chaos
\newblock {\em Nonlinearity} {\bf 34} 5744--5761

\bibitem{seidel_1933}
Seidel W 1933
\newblock Note on a metrically transitive system
\newblock {\em Proc. Nat. Acad. Sci. USA} {\bf 19} 453--456

\bibitem{weisstein04}
Weisstein E W
\newblock Series {M}ultisection
\newblock {\em From MathWorld- A Wolfram Web Resource}
\newblock \url{https://mathworld.wolfram.com/SeriesMultisection.html}

\bibitem{zammert14}
Zammert S and Eckhardt B 2014
\newblock Streamwise and doubly-localised periodic orbits in plane 
Poiseuille flow
\newblock {\em Journal of Fluid Mechanics} {\bf 761} 348–359

\end{thebibliography}
\end{document}